\documentclass[12pt]{amsart}

\usepackage{amsmath,amssymb,enumerate}
\newtheorem{theorem}{Theorem}[section]

\newtheorem{lemma}[theorem]{Lemma}

\newtheorem{corollary}[theorem]{Corollary}
\newtheorem{conjecture}[theorem]{Conjecture}
\theoremstyle{definition}

\newcommand{\cF}{\mathcal{F}}

\newcommand{\cM}{\mathcal{M}}

\newcommand{\cU}{\mathcal{U}}

\DeclareMathOperator{\si}{si}

\DeclareMathOperator{\cl}{cl}

\DeclareMathOperator{\PG}{PG}

\DeclareMathOperator{\GF}{GF}
\DeclareMathOperator{\AG}{AG}
\newcommand{\elem}{\varepsilon}
\newcommand{\del}{\!\setminus\!}
\newcommand{\con}{/}

\begin{document}
\sloppy

\title[Density Hales-Jewett for Matroids]{A Density Hales-Jewett Theorem for Matroids}

\author[Geelen]{Jim Geelen}
\address{Department of Combinatorics and Optimization,
University of Waterloo, Waterloo, Canada}
\thanks{ This research was partially supported by a grant from the
Office of Naval Research [N00014-12-1-0031].}

\author[Nelson]{Peter Nelson}
\address{School of Mathematics, Statistics and Operations Research, Victoria University of Wellington, Wellington, New Zealand}


\date{\today}

\begin{abstract}
	We show that, if $\alpha > 0$ is a real number, $n \ge 2$ and $\ell \ge 2$ are integers, and $q$ is a prime power, then every simple matroid $M$ of sufficiently large rank, with no $U_{2,\ell}$-minor, no rank-$n$ projective geometry minor over a larger field than $\GF(q)$, and satisfying $|M| \ge \alpha q^{r(M)}$, has a rank-$n$ affine geometry restriction over $\GF(q)$. This result can be viewed as an analogue of the multidimensional density Hales-Jewett theorem for matroids. 
\end{abstract}

\maketitle

\section{Introduction}

	Furstenberg and Katznelson [\ref{fk}] proved the following result, implying that $\GF(q)$-representable matroids of nonvanishing density and huge rank contain large affine geometries as restrictions:

\begin{theorem}\label{dhjnice}
	Let $q$ be a prime power, $\alpha > 0$ be a real number, and $n$ be a positive integer. If $M$ is a simple $\GF(q)$-representable matroid with $|M| \ge \alpha q^{r(M)}$, and $M$ has sufficiently large rank, then $M$ has an $\AG(n,q)$-restriction. 
\end{theorem}

Later, Furstenberg and Katznelson [\ref{fk2}] proved a much more general result; namely the Multidimensional Density Hales-Jewett Theorem, which gives a similar statement in the more abstract setting of words over an arbitrary finite alphabet. Considerably shorter proofs [\ref{dkt},\ref{polymath}] have since been found. We will generalise Theorem~\ref{dhjnice} in a different direction:

\begin{theorem}\label{maineasy}
	Let $q$ be a prime power, $n$ be a positive integer, and $\alpha > 0$ be a real number. If $M$ is a simple matroid with no $U_{2,q+2}$-minor and $|M| \ge \alpha q^{r(M)}$, and $M$ has sufficiently large rank, then $M$ has an $\AG(n,q)$-restriction. 
\end{theorem}

In fact, we prove more. The class of matroids with no $U_{2,q+2}$-minor is just one of many minor-closed classes whose extremal behaviour is qualitatively similar to that of the $\GF(q)$-representable matroids. The following theorem, which summarises several papers [\ref{gkp},\ref{gkw},\ref{gw}], tells us that such classes occur naturally as one of three types:

\begin{theorem}[Growth Rate Theorem]\label{grt}
	Let $\cM$ be a minor-closed class of matroids, not containing all simple rank-$2$ matroids. There exists a real number $c_{\cM} > 0$ such that either:
	\begin{enumerate}[(1)]
		\item $|M| \le c_{\cM}r(M)$ for every simple $M \in \cM$, 
		\item $|M| \le c_{\cM}r(M)^2$ for every simple $M \in \cM$, and $\cM$ contains all graphic matroids, or 
		\item there is a prime power $q$ such that $|M| \le c_{\cM}q^{r(M)}$ for every simple $M \in \cM$, and $\cM$ contains all $\GF(q)$-representable matroids. 
	\end{enumerate}
\end{theorem}

We call a class $\cM$ satisfying (3) \emph{base-$q$ exponentially dense}. It is clear that these classes are the only ones that contain arbitrarily large affine geometries, and that the matroids with no $U_{2,q+2}$-minor form such a class. Our main result, which clearly implies Theorem~\ref{maineasy}, is the following: 

\begin{theorem}\label{main}
	Let $\cM$ be a base-$q$ exponentially dense minor-closed class of matroids, $\alpha > 0$ be a real number, and $n$ be a positive integer.  If $M \in \cM$ is simple, satisfies $|M| \ge \alpha q^{r(M)}$, and has sufficiently large rank, then $M$ has an $\AG(n,q)$-restriction. 
\end{theorem}

Finding such a highly structured restriction seems very surprising, given the apparent wildness of general exponentially dense classes. This will be proved using Theorem~\ref{grt} and a slightly more technical statement, Theorem~\ref{firsthalf}; the proof extensively uses machinery developed in [\ref{gn}], [\ref{gn2}], [\ref{sqf}] and [\ref{thesis}].

We would like to prove a result corresponding to Theorem~\ref{main} for \emph{quadratically dense} classes satisfying condition (2) of Theorem~\ref{grt}. The following is a corollary of the Erd\H os-Stone Theorem [\ref{es}]:

\begin{theorem}
	Let $\alpha > 0$ be a real number and $n$ be a positive integer. If $G$ is a simple graph such that $|E(G)| \ge \alpha |V(G)|^2$ and $|V(G)|$ is sufficiently large, then $G$ has a $K_{n,n}$-subgraph. 
\end{theorem}

In light of this, we expect that the unavoidable restrictions of dense matroids in a quadratically dense class are the cycle matroids of large complete bipartite graphs. 
\begin{conjecture}
	Let $\cM$ be a quadratically dense minor-closed class of matroids, $\alpha > 0$ be a real number, and $n$ be a positive integer. If $M \in \cM$ is simple, satisfies $|M| \ge \alpha r(M)^2$, and has sufficiently large rank, then $M$ has an $M(K_{n,n})$-restriction. 
\end{conjecture}

\section{Preliminaries}

We follow the notation of Oxley [\ref{oxley}]. For a matroid $M$, we also write $|M|$ for $|E(M)|$, and $\elem(M)$ for $|\si(M)|$, or the number of points in $M$. If $\ell \ge 2$ is an integer, we write $\cU(\ell)$ for the class of matroids with no $U_{2,\ell+2}$-minor.

The next theorem, a constituent of Theorem~\ref{grt}, follows easily from the two main results of [\ref{gkp}]. 

\begin{theorem}\label{gk}
	There is a real-valued function $\alpha_{\ref{gk}}(n,\gamma,\ell)$ so that, if $\ell \ge 2$ and $n \ge 2$ are integers, $\gamma > 1$ is a real number, and $M \in \cU(\ell)$ satisfies $\elem(M) \ge \alpha_{\ref{gk}}(n,\gamma,\ell)\gamma^{r(M)}$, then $M$ has a $\PG(n-1,q)$-minor for some $q > \gamma$. 
\end{theorem}

The next theorem is due to Kung [\ref{kungextremal}].

\begin{theorem}\label{kung}
	If $\ell \ge 2$ and $M \in \cU(\ell)$, then $\elem(M) \le \frac{\ell^{r(M)}-1}{\ell-1}$. 
\end{theorem}

We will sometimes use the cruder estimate $\elem(M) \le (\ell+1)^{r(M)-1}$ for ease of notation, such as in the following easy corollary: 

\begin{corollary}\label{kungrel}
	If $\ell \ge 2$ is an integer, $M \in \cU(\ell)$, and $C \subseteq E(M)$, then $\elem(M \con C) \ge (\ell+1)^{-r_M(C)}\elem(M)$.
\end{corollary}
\begin{proof}
	Let $\cF$ be the collection of rank-$(r_M(C) + 1)$ flats of $M$ containing $C$. We have $\elem(M|F) \le \frac{\ell^{r_M(C)+1}-1}{\ell-1} \le (\ell+1)^{r_M(C)}$ for each $F \in \cF$. Moreover, $|\cF| = \elem(M \con C)$, and $\elem(M) \le \sum_{F \in \cF}\elem(M|F)$; the result follows.
\end{proof}

We apply both the above results freely. 

The next result follows from [\ref{gn2}, Lemma 3.1]. 

\begin{lemma}\label{projectiondensity}
	Let $q$ be a prime power, $k \ge 0$ be an integer, and $M$ be a matroid with a $\PG(r(M)-1,q)$-restriction $R$. If $F$ is a rank-$k$ flat of $M$ that is disjoint from $E(R)$, then $\elem(M \con F) \ge \frac{q^{r(M \con F) + k}-1}{q-1} - q\frac{q^{2k}-1}{q^2-1}$. 
\end{lemma}

\section{Connectivity}

	A matroid $M$ is \emph{weakly round} if there is no pair of sets $A,B$ with union $E(M)$, such that $r_M(A) \le r(M)-1$ and $r_M(B) \le r(M)-2$. This is a variation on \emph{roundness}, a notion equivalent to infinite vertical connectivity introduced by Kung in [\ref{kungroundness}] under the name of \emph{non-splitting}. Our tool for reducing Theorem~\ref{main} to the weakly round case is the following, proved in [\ref{sqf}, Lemma 7.2]. 

	\begin{lemma}\label{weakroundnessreduction}
		There is an integer-valued function $f_{\ref{weakroundnessreduction}}(r,d,\ell)$ so that, for any integers $\ell \ge 2$ and $0 \le d \le r$, and real-valued function $g(n)$ satisfying $g(d) \ge 1$ and $g(n) \ge 2g(n-1)$ for all $n > d$, if $M \in \cU(\ell)$ satisfies $r(M) \ge f_{\ref{weakroundnessreduction}}(r,d,\ell)$ and $\elem(M) > g(r(M))$, then $M$ has a weakly round restriction $N$ such that $r(N) \ge r$ and $\elem(N) > g(r(N))$. 
	\end{lemma}
	
	Our next lemma, proved in [\ref{gn2}, Lemma 8.1], allows us to exploit weak roundness by contracting an interesting low-rank restriction onto a projective geometry. 
	
	\begin{lemma}\label{contractrestriction}
There is an integer-valued function $f_{\ref{contractrestriction}}(n,q,t,\ell)$
so that, for any prime power $q$ and integers $n \ge 1,\ell \ge 2$ and $t \ge
0$, if $M \in \cU(\ell)$ is weakly round and has a
$\PG(f_{\ref{contractrestriction}}(n,q,t,\ell)-1,q)$-minor and a restriction $T$ of rank at most $t$, then there is a minor $N$ of $M$ of
rank at least $n$, such that $T$ is a restriction of $N$, and $N$ 
has a $\PG(r(N)-1,q)$-restriction.
	\end{lemma}

\section{Stacks}

	We now define an obstruction to $\GF(q)$-representability. If $q$ is a prime power, and $h$ and $t$ are nonnegative integers, then a matroid $S$ is a \emph{$(q,h,t)$-stack} if there are pairwise disjoint subsets $F_1, F_2, \dotsc, F_h$ of $E(S)$ such that the union of the $F_i$ is spanning in $F$, and for each $i \in \{1, \dotsc, h\}$, the matroid $(S \con (F_1 \cup \dotsc \cup F_{i-1}))|F_i$ has rank at most $t$ and is not $\GF(q)$-representable. We write $F_i(S)$ for $F_i$. Note that such a stack has rank at most $ht$. When the value of $t$ is unimportant, we refer simply to a \emph{$(q,h)$-stack}.  

	The next three results suggest that stacks are `incompatible' with large projective geometries. First we argue that a matroid obtained from a projective geometry by applying a small `projection' does not contain a large stack:
	
	\begin{lemma}\label{stackinprojection}
		Let $q$ be a prime power and $h$ be a nonnegative integer. If $M$ is a matroid and $X \subseteq E(M)$ satisfies $r_M(X) \le h$ and $\si(M \del X) \cong \PG(r(M)-1,q)$, then $M \con X$ has no $(q,h+1)$-stack restriction. 
	\end{lemma}
	\begin{proof}
		The result is clear if $h=0$; suppose that $h > 0$ and that the result holds for smaller $h$. Moreover suppose for a contradiction that $M \con X$ has a $(q,h+1,t)$-stack restriction $S$. Let $F = F_1(S)$. Since $(M \con X)|F$ is not $\GF(q)$-representable but $M|F$ is, it follows that $\sqcap_{M}(F,X) > 0$. Therefore  $r_{M \con F}(X) < r_M(X) \le h$ and $\si(M \con F \del X) \cong \PG(r(M \con F)-1,q)$, so by the inductive hypothesis $M \con (X \cup F)$ has no $(q,h)$-stack restriction. Since $M \con (X \cup F)|(E(S)-F)$ is clearly such a stack, this is a contradiction.  
	\end{proof}
	
	Now we show that a large stack on top of a projective geometry allows us to find a large flat disjoint from the geometry:
	
	\begin{lemma}\label{stackfindprojection}
		Let $q$ be a prime power and $h$ be a nonnegative integer. If $M$ is a matroid with a $\PG(r(M)-1,q)$-restriction $R$ and a $(q,\binom{h+1}{2})$-stack restriction $S$, then $E(S) - E(R)$ contains a rank-$h$ flat of $M$. 
	\end{lemma}
	\begin{proof}
		If $h = 0$, then there is nothing to show; suppose that $h > 0$ and that the result holds for smaller $h$. Let $S$ be a $(q,\binom{h+1}{2})$-stack restriction of $M$ and let $F_i = F_i(S)$ for each $i \in \{1, \dotsc, \binom{h+1}{2}\}$. Let $S_1 = S| F_1 \cup \dotsc \cup F_{\binom{h}{2}}$. Clearly $S_1$ is a $(q,\binom{h}{2})$-stack, so inductively $E(S_1)-E(R)$ contains a rank-$(h-1)$ flat $H$ of $M$. 
		
		Note that $(M \con H)|E(R)$ has no loops. If $M \con H$ has a nonloop $e$ that is not parallel to an element of $R$, then $\cl_M(H \cup \{e\})$ is a rank-$h$ flat of $M$ contained in $E(S)-E(R)$, and we are done. Therefore we may assume that $\si(M \con H) \cong \si((M \con H)|E(R))$, and so by Lemma~\ref{stackinprojection} applied to the matroid $M|(E(R) \cup H)$, we know that $M \con H$ has no $(q,h)$-stack restriction. However the sets $(E(S_1)-H) \cup F_{\binom{h}{2}+1}, F_{\binom{h}{2}+2}, \dotsc, F_{\binom{h+1}{2}}$ clearly give rise to such a stack. This is a contradiction. 
		
	\end{proof}
	
	Finally we show that a large stack restriction, together with a very large projective geometry minor, gives a projective geometry minor over a larger field:
	
	\begin{lemma}\label{stackwin}
		There are integer-valued functions $f_{\ref{stackwin}}(n,q,t,\ell)$ and $h_{\ref{stackwin}}(n,q,\ell)$ so that, for each prime power $q$ and any integers $\ell \ge 2$, $n \ge 2$ and $t \ge 0$, if $M \in \cU(\ell)$ is weakly round with a $\PG(f_{\ref{stackwin}}(n,q,t,\ell)-1,q)$-minor and a $(q,h_{\ref{stackwin}}(n,q,\ell),t)$-stack restriction, then $M$ has a $\PG(n-1,q')$-minor for some $q' > q$. 
	\end{lemma}
	\begin{proof}
		Let $q$ be a prime power and $\ell \ge 2$, $n \ge 2$ and $t \ge 0$ be integers. Let $\alpha = \alpha_{\ref{gk}}(n,q,\ell)$, and let $h' > 0$ and $r \ge 0$ be integers so that $\frac{q^{r'+h'}-1}{q-1} - q\frac{q^{2h'}-1}{q^2-1} > \alpha q^{r'}$ for all $r' \ge r$. Set $h_{\ref{stackwin}}(n,q,\ell) = h = \binom{h'+1}{2}$, and $f_{\ref{stackwin}}(n,q,t,\ell) = f_{\ref{contractrestriction}}(r+h',q,th,\ell)$. 
		
		Let $M \in \cU(\ell)$ be weakly round with a $\PG(f_{\ref{stackwin}}(n,q,t,\ell)-1,q)$-minor and an $(q,h,t)$-stack restriction $S$. We have $r(S) \le th$; by Lemma~\ref{contractrestriction} there is a minor $N$ of $M$, of rank at least $r+h'$, with a $\PG(r(N)-1,q)$-restriction $R$, and $S$ as a restriction. By Lemma~\ref{stackfindprojection}, $E(S)-E(R)$ contains a rank-$h'$ flat $F$ of $M$. Now $r(M \con F) \ge r$; the lemma follows from Lemma~\ref{projectiondensity}, Theorem~\ref{gk}, and the definition of $h'$. 
	\end{proof}
	
\section{Lifting}
	
	The following is a restatement of Theorem~\ref{dhjnice}:

	\begin{theorem}\label{dhj}
		There is an integer-valued function $f_{\ref{dhj}}(n,q,\beta)$ so that, if $\beta > 0$ is a real number, $q$ is a prime power, and $M$ is a $\GF(q)$-representable matroid satisfying $\elem(M) \ge \beta q^{r(M)}$ and $r(M) \ge f_{\ref{dhj}}(n,q,\beta)$, then $M$ has an $AG(n-1,q)$-restriction. 
	\end{theorem}

	This next lemma uses the above to show that a bounded lift of a huge affine geometry itself contains a large affine geometry. The proof does not use the full strength of ~\ref{dhj}; the lemma would also follow from the much weaker `colouring' Hales-Jewett Theorem [\ref{hj}]. 
	\begin{lemma}\label{liftag}
		There is an integer-valued function $f_{\ref{liftag}}(n,q,\ell,t)$ so that, if $\ell \ge 2$ and $n \ge 2$ and $t \ge 0$ are integers, $q$ is a prime power, $M \in \cU(\ell)$ and $C \subseteq E(M)$ satisfy $r_M(C) \le t$, and $M \con C$ has an $\AG(f_{\ref{liftag}}(n,q,\ell,t)-1,q)$-restriction, then $M$ has an $\AG(n-1,q)$-restriction. 
	\end{lemma}
	\begin{proof}
		Let $\ell \ge 2$, $n \ge 2$ and $t \ge 0$ be integers, and $q$ be a prime power. Let $d$ be an integer large enough so that $(\ell+1)^{-t} > \frac{q^{2-d}}{q-1}$, and let $m = f_{\ref{dhj}}(n,q,(q^2 (\ell+1)^t)^{-1}) + d$. Set $f_{\ref{liftag}}(n,q,\ell,t) = m$. 
	
		Let $M \in \cU(\ell)$, and let $C \subseteq E(M)$ be a set so that $r_M(C) \le t$ and $M \con C$ has an $\AG(m-1,q)$-restriction $R$. We may assume that $C$ is independent and that $E(M) = E(R) \cup C$, so $M$ is simple and $r(M) = m + |C|$. Let $B$ be a basis for $M$ containing $C$, and let $e \in B - C$. Let $X = B - (C \cup \{e\})$. Now $\cl_{M \con C}(X)$ is a hyperplane of $R$, so $|\cl_{M \con C}(X)| = q^{m-2}$ and there are at least $q^{m-1} - q^{m-2} \ge q^{m-2}$ elements of $M$ not spanned by $X \cup C$. Each such element lies in a point of $M \con X$ and is not spanned by $C$ in $M \con X$. Moreover, $r(M \con X) = t+1$, so $M \con X$ has at most $(\ell+1)^t$ points; there is thus a point $P$ of $M \con X$, not spanned by $C$, with $|P| \ge (\ell+1)^{-t}q^{m-2}$. 
		
		Now $P \subseteq E(R)$, so the matroid $(M \con C)|P$ is $\GF(q)$-representable and has rank at most $m$, and $\elem((M \con C)|P) \ge (\ell+1)^{-t}q^{m-2} > \frac{q^{m-d}-1}{q-1}$, so $r((M \con C)|P) \ge m - d$. Furthermore, $\elem((M\con C)|P) \ge (q^2(\ell+1)^t)^{-1}q^m \ge (q^2(\ell+1)^t)^{-1} q^{r((M \con C)|P)}$, so by Theorem~\ref{dhj} and the definition of $m$, the matroid $(M \con C)|P$ has an $\AG(n-1,q)$-restriction. However, $P$ is skew to $C$ in $M$ by construction, so $(M \con C)|P = M|P$, and therefore $M$ also has an $\AG(n-1,q)$-restriction, as required.
	\end{proof}
	
\section{The Main Result}

	Since, for any base-$q$ exponentially dense minor-closed class $\cM$, there is some $\ell \ge 2$ such that $\cM \subseteq \cU(\ell)$ and there is some $s$ such that $\PG(s,q') \notin \cM$ for all $q' > q$, this next theorem easily implies Theorem~\ref{main}. 

	\begin{theorem}\label{firsthalf}
		There is an integer-valued function $f_{\ref{firsthalf}}(n,q,\ell,\beta)$ so that if $\beta > 0$ is a real number, $n \ge 2$ and $\ell \ge 2$ are integers, $q$ is a prime power, and $M \in \cU(\ell)$ satisfies $r(M) \ge f_{\ref{firsthalf}}(n,q,\ell,\beta)$ and $\elem(M) \ge \beta q^{r(M)}$, then $M$ has either an $\AG(n-1,q)$-restriction or a $\PG(n-1,q')$-minor for some $q' > q$. 
	\end{theorem}
	\begin{proof}
		Let $\beta > 0$ be a real number, $q$ be a prime power, and $\ell \ge 2$ and $n \ge 2$ be integers. Let $\alpha = \alpha_{\ref{gk}}(n,q,\ell)$, $h = h_{\ref{stackwin}}(n,q,\ell)$. Set $0 = t_0, t_1, \dotsc, t_h$ to be a nondecreasing sequence of integers such that \[t_{k+1} \ge f_{\ref{dhj}}(f_{\ref{liftag}}(n,q,\ell,kt_k),q,\beta((\ell+1)^{kt_k}q\alpha)^{-1})\] for each $k \in \{0, \dotsc, h-1\}$. Let $m = \max(n,f_{\ref{stackwin}}(n,q,\ell,t_h))$, and let $r_1$ be an integer large enough so that $r_1 \ge (h+1)t_h$, $q^{(h+1)t_h - r_1 - 1} \le \alpha$, and $\beta q^{r} \ge \alpha_{\ref{gk}}(m,q-\tfrac{1}{2},\ell)(q-\tfrac{1}{2})^{r}$ for all $r \ge r_1$. Let $d$ be an integer such that $\beta q^d \ge 1$, and let $r_2 = f_{\ref{weakroundnessreduction}}(r_1,d,\ell)$. 
		
		Let $M_2 \in \cU(\ell)$ satisfy $r(M_2) \ge r_2$ and $\elem(M_2) \ge \beta q^{r(M_2)}$; we will show that $M_2$ has either a $\PG(n-1,q')$-minor for some $q' > q$, or an $\AG(n-1,q)$-restriction. The function $g(r) = \beta q^{r}$ satisfies $g(d) \ge 1$ and $g(r) \ge 2g(r-1)$ for all $r > d$, so by Lemma~\ref{weakroundnessreduction} $M_2$ has a weakly round restriction $M_1$ such that $r(M_1) \ge r_1$ and $\elem(M_1) \ge \beta q^{r(M_1)}$.
				
		 Let $k \ge 0$ be maximal such that $k \le h$ and $M_1$ has a $(q,k,t_k)$-stack restriction $S$. We split into cases depending on whether $k = h$:
		
		\textbf{Case 1: $k = h$}.
		
		Note that $\elem(M_1) \ge \beta q^{r(M_1)} \ge \alpha_{\ref{gk}}(m,q-\tfrac{1}{2},\ell)(q-\tfrac{1}{2})^{r(M_1)}$, so $M_1$ has a $\PG(m-1,q')$-minor for some $q' > q - \tfrac{1}{2}$. If $q' > q$, then we have the first outcome, since $m > n$. Therefore we may assume that $M_1$ has a $\PG(m-1,q)$-minor. $M_1$ also has a $(q,h,t_h)$-stack restriction, and the first outcome now follows from Lemma~\ref{stackwin} and the definitions of $m$ and $h$.
		
		\textbf{Case 2: $k < h$}. 
		
		Let $M_0 = \si(M_1 \con E(S))$; note that $r(M_0) \ge r(M_1) - kt_k$, and therefore that $|M_0| \ge (\ell+1)^{-kt_k}|M_1| \ge (\ell+1)^{-kt_k}\beta q^{r(M_0)}$. Let $F_0$ be a rank-$(t_{k+1}-1)$ flat of $M_0$, and consider the matroid $M_0 \con F_0$. If $\elem(M_0 \con F_0) \ge \alpha q^{r(M_0 \con F_0)}$, then we have the first outcome by Theorem~\ref{gk}, so we may assume that $\elem(M_0 \con F_0) \le \alpha q^{r(M_0 \con F_0)} = \alpha q^{r(M_0) - t_{k+1} + 1}$. Let $\cF$ be the collection of rank-$t_{k+1}$ flats of $M_0$ containing $F_0$. By a majority argument, there is some $F \in \cF$ satisfying 
		\begin{align*}
			|F-F_0| &\ge |\cF|^{-1}|M_0 \del F_0| \\
					&\ge \elem(M_0 \con F_0)^{-1} ((\ell+1)^{-kt_k}\beta q^{r(M_0)} - |F_0|) \\
					&\ge \alpha^{-1}q^{-r(M_1)+t_{k+1}-1}((\ell+1)^{-kt_k}\beta q^{r(M_0)}- |F_0|). 
		\end{align*}
		Since $\alpha^{-1}q^{-r(M_1) + r(S) + t_{k+1}-1} \le \alpha^{-1}q^{-r_1 + (h+1)t_h-1} \le 1$, this gives $|F| \ge \alpha^{-1}q^{-r(M_0) + t_{k+1} - 1} (\ell+1)^{-kt_k}\beta q^{r(M_0)} = \beta((\ell+1)^{kt_k} q\alpha)^{-1}q^{r(M_0|F)}$. By maximality of $k$, we know that $M_0|F$ is $\GF(q)$-representable, and $r(M_0|F) = t_{k+1} \ge f_{\ref{dhj}}(f_{\ref{liftag}}(n,q,\ell,kt_k),q,\beta((\ell+1)^{kt_k}q\alpha)^{-1})$, so $M_0|F$ has an $\AG(f_{\ref{liftag}}(n,q,\ell,kt_k)-1,q)$-restriction by Theorem~\ref{dhj}. Now $M_0 = \si(M_1 \con E(S))$ and $r(S) \le kt_k$, so by Lemma~\ref{liftag}, $M_1$ has an $\AG(n-1,q)$-restriction, and so does $M_2$.

	\end{proof}

\section*{References}
\newcounter{refs}
\begin{list}{[\arabic{refs}]}
{\usecounter{refs}\setlength{\leftmargin}{10mm}\setlength{\itemsep}{0mm}}

\item\label{dkt}
P. Dodos, V. Kanellopoulos, K. Tyros, 
A simple proof of the density Hales-Jewett theorem, 
arXiv:1209.4986v1 [math.CO], (2012) 1-11.

\item\label{es}
P. Erd\H os, A.H. Stone,
On the structure of linear graphs,
Bull. Amer. Math. Soc. 52 (1946), 1087--1091.

\item\label{fk}
H. Furstenberg, Y. Katznelson,
IP-sets, Szemer\'edi's Theorem and Ramsey Theory,
Bull. Amer. Math. Soc. (N.S.) 14 no. 2 (1986), 275--278.

\item\label{fk2}
H. Furstenberg, Y. Katznelson, 
A density version of the Hales-Jewett Theorem, 
J. Anal. Math. 57 (1991), 64--119.

\item\label{gkp}
J. Geelen, K. Kabell,
Projective geometries in dense matroids, 
J. Combin. Theory Ser. B 99 (2009), 1--8.

\item\label{gkw}
J. Geelen, J.P.S. Kung, G. Whittle,
Growth rates of minor-closed classes of matroids,
J. Combin. Theory. Ser. B 99 (2009), 420--427.

\item\label{gn}
J. Geelen, P. Nelson, 
The number of points in a matroid with no $n$-point line as a minor, 
J. Combin. Theory. Ser. B 100 (2010), 625--630.

\item\label{gn2}
J. Geelen, P. Nelson, 
On minor-closed classes of matroids with exponential growth rate, 
Adv. Appl. Math. (2012), to appear. 

\item\label{gw}
J. Geelen, G. Whittle,
Cliques in dense $\GF(q)$-representable matroids, 
J. Combin. Theory. Ser. B 87 (2003), 264--269.

\item\label{hj}
A. Hales and R. Jewett,
Regularity and positional games,
Trans. Amer. Math. Soc., 106(2) (1963), 222--229.

\item\label{kungextremal}
J.P.S. Kung,
Extremal matroid theory, in: Graph Structure Theory (Seattle WA, 1991), 
Contemporary Mathematics 147 (1993), American Mathematical Society, Providence RI, ~21--61.

\item\label{kungroundness}
J.P.S. Kung, 
Numerically regular hereditary classes of combinatorial geometries,
Geom. Dedicata 21 (1986), no. 1, 85--10.

\item\label{polymath}
D.H.J. Polymath,
A new proof of the density Hales-Jewett theorem,
arXiv:0910.3926v2 [math.CO], (2010) 1-34.

\item\label{sqf}
P. Nelson,
Growth rate functions of dense classes of representable matroids, 
J. Combin. Theory. Ser. B (2012), in press.

\item\label{thesis}
P. Nelson,
Exponentially Dense Matroids,
Ph.D thesis, University of Waterloo (2011). 

\item \label{oxley}
J. G. Oxley, 
Matroid Theory,
Oxford University Press, New York (2011).
\end{list}		
\end{document}